\documentclass[11pt,twoside]{preprint}
\usepackage[a4paper,innermargin=1.4in,outermargin=1.4in,bottom=1.5in,marginparwidth=1.5in,marginparsep=3mm]{geometry}
\usepackage{amsmath,amsthm,amssymb,enumerate}
\usepackage{hyperref}
\usepackage{mathrsfs}
\usepackage{breakurl}
\usepackage{tikz}

\usepackage[osf]{newtxtext}
\usepackage[varqu,varl]{inconsolata}
\usepackage[cal=boondoxo]{mathalfa}
\usepackage[bigdelims,vvarbb,cmintegrals]{newtxmath}

\usepackage{Symbols}
\usepackage{bbm}

\usepackage{mhequ}
\usepackage{comment}
\usepackage{microtype}
\usepackage{comment}
\usepackage{xcolor}
\usepackage{subfig}

\colorlet{darkblue}{blue!90!black}
\colorlet{darkred}{red!90!black}



\renewcommand{\theequation}{\thesection.\arabic{equation}}

\newtheorem{theorem}{Theorem}[section]

\newtheorem{proposition}[theorem]{Proposition}
\newtheorem{corollary}[theorem]{Corollary}
\newtheorem{conjecture}{Conjectrure}[section]
\theoremstyle{definition}

\newtheorem{definition}[theorem]{Definition}

\newtheorem{remark}[theorem]{Remark}

\def\emptyset{\varnothing}

\newcommand{\bxi}{\boldsymbol\xi}

\newcommand{\poly}{\mathrm{poly}}

\newcommand{\E}{\mathbf E}
\newcommand{\bR}{\mathbf{R}}

\renewcommand{\P}{\mathbf P}
\newcommand{\MM}{\mathscr{M}}
\newcommand{\ST}{\mathscr{T}}

\newcommand{\SD}{\mathscr{D}}
\newcommand{\LL}{\mathscr{L}}

\newcommand{\ch}{\mathcal{h}}
\newcommand{\KPZ}{\mathrm{KPZ}}

\definecolor{darkred}{rgb}{0.7,0.1,0.1}
\definecolor{darkblue}{rgb}{0.1,0.1,0.8}

\def\s{\mathfrak s}

\def\T{\mathbb{T}}

\def\${|\!|\!|}
\def\deg{\mathop{\mathrm{deg}}}

\def\1{\mathbbm{1}}

\usetikzlibrary{shapes.misc}
\usetikzlibrary{shapes.symbols}
\tikzset{
	doth/.style={circle,fill=white,draw=black, solid,inner sep=0pt,minimum size=0.5mm},
	yy/.style={circle,fill=gray!20,draw=black,inner sep=0pt,minimum size=0.8mm},
	>=stealth,
	}
\tikzset{
	root/.style={circle,fill=testcolor,inner sep=0pt, minimum size=2mm},
	dot/.style={circle,fill=black,inner sep=0pt, minimum size=1mm},
	smalldot/.style={circle,fill=black,inner sep=0pt, minimum size=0.5mm},
	edge/.style={thick},
	edged/.style={thick, densely dotted},
	var/.style={circle,fill=black!10,draw=black,inner sep=0pt, minimum size=2mm},
	circ/.style={circle,fill=white,draw=black,inner sep=0pt, minimum size=1.2mm},
	dotred/.style={circle,fill=black!50,inner sep=0pt, minimum size=2mm},
	generic/.style={semithick,shorten >=1pt,shorten <=1pt},
	gepsilon/.style={semithick,shorten >=1pt,shorten <=1pt,densely dashed},
	dist/.style={ultra thick,draw=testcolor,shorten >=1pt,shorten <=1pt},
	testfcn/.style={ultra thick,testcolor,shorten >=1pt,shorten <=1pt,<-},
	testfcnx/.style={ultra thick,testcolor,shorten >=1pt,shorten <=1pt,<-,
		postaction={decorate,decoration={markings,mark=at position 0.6 with {\drawx}}}},
	kprime/.style={semithick,shorten >=1pt,shorten <=1pt,dotted,->},
	kprimex/.style={semithick,shorten >=1pt,shorten <=1pt,densely dashed,->,
		postaction={decorate,decoration={markings,mark=at position 0.4 with {\drawx}}}},
	kernel/.style={semithick,shorten >=1pt,shorten <=1pt,->},
	multx/.style={shorten >=1pt,shorten <=1pt,
		postaction={decorate,decoration={markings,mark=at position 0.5 with {\drawx}}}},
	kernelx/.style={semithick,shorten >=1pt,shorten <=1pt,->,
		postaction={decorate,decoration={markings,mark=at position 0.4 with {\drawx}}}},
	kepsilon/.style={semithick,shorten >=1pt,shorten <=1pt,densely dashed,->},
	kernel1/.style={->,semithick,shorten >=1pt,shorten <=1pt,postaction={decorate,decoration={markings,mark=at position 0.45 with {\draw[-] (0,-0.1) -- (0,0.1);}}}},
	kernel2/.style={->,semithick,shorten >=1pt,shorten <=1pt,postaction={decorate,decoration={markings,mark=at position 0.45 with {\draw[-] (0.05,-0.1) -- (0.05,0.1);\draw[-] (-0.05,-0.1) -- (-0.05,0.1);}}}},
	kernelBig/.style={semithick,shorten >=1pt,shorten <=1pt,decorate, decoration={zigzag,amplitude=1.5pt,segment length = 3pt,pre length=2pt,post length=2pt}},
	rho/.style={dotted,semithick,shorten >=1pt,shorten <=1pt},
	renorm/.style={shape=circle,fill=white,inner sep=1pt},
	labl/.style={shape=rectangle,fill=white,inner sep=1pt},
	xi/.style={circle,fill=symbols!10,draw=symbols,inner sep=0pt,minimum size=1.2mm},
	xix/.style={crosscircle,fill=symbols!10,draw=symbols,inner sep=0pt,minimum size=1.2mm},
	xib/.style={circle,fill=symbols!10,draw=symbols,inner sep=0pt,minimum size=1.6mm},
	xibx/.style={crosscircle,fill=symbols!10,draw=symbols,inner sep=0pt,minimum size=1.6mm},
	not/.style={circle,fill=symbols,draw=symbols,inner sep=0pt,minimum size=0.5mm},
	>=stealth,
	}
\makeatletter
\def\DeclareSymbol#1#2#3{\expandafter\gdef\csname MH@symb@#1\endcsname{\tikz[baseline=#2,scale=0.15]{#3}}%
\expandafter\gdef\csname MH@symb@#1s\endcsname{\scalebox{0.6}{\tikz[baseline=#2,scale=0.15]{#3}}}}
\def\<#1>{\csname MH@symb@#1\endcsname}
\makeatother

\DeclareSymbol{X}{-2.4}{\node[dot] {};}
\DeclareSymbol{edge}{0}{\draw[white] (-.4,0) -- (.4,0); \draw[edge] (0,0)  -- (0,1.5) {};}
\DeclareSymbol{edged}{0}{\draw[white] (-.4,0) -- (.4,0); \draw[edged] (0,0)  -- (0,1.5) {};}
\DeclareSymbol{1}{0}{\draw[white] (-.4,0) -- (.4,0); \draw[edge] (0,0)  -- (0,1.5) node[dot] {};}
\DeclareSymbol{1d}{0}{\draw[white] (-.4,0) -- (.4,0); \draw[edged] (0,0)  -- (0,1.5) node[dot] {};}
\DeclareSymbol{2}{0}{\draw (-0.5,1.2) node[dot] {} -- (0,0) node[smalldot] {} -- (0.5,1.2) node[dot] {};}
\DeclareSymbol{2d}{0}{\draw[edged] (-0.7,1.5) node[dot] {} -- (0,0) node[smalldot] {} -- (0.7,1.5) node[dot] {};}
\DeclareSymbol{2d1}{0}{\draw[edged] (-0.7,2) node[dot] {} -- (0,0.5) -- (0.7,2) node[dot] {}; \draw[edge] (0,0.5) node[smalldot] {} -- (0,-0.7);}
\DeclareSymbol{1d1}{0}{\draw[edged] (0,0.5) -- (0.7,2) node[dot] {}; \draw[edge] (0,0.5) node[smalldot] {} -- (0,-0.7);}
\DeclareSymbol{1d1d}{0}{\draw[edged] (0,0.5) -- (0.7,2) node[dot] {}; \draw[edged] (0,0.5) node[smalldot] {} -- (0,-0.7);}
\DeclareSymbol{2d1}{0}{\draw[edged] (-0.7,2) node[dot] {} -- (0,0.5) -- (0.7,2) node[dot] {}; \draw[edge] (0,0.5) node[smalldot] {} -- (0,-0.7);}
\DeclareSymbol{2d1d}{0}{\draw[edged] (-0.7,2) node[dot] {} -- (0,0.5) -- (0.7,2) node[dot] {}; \draw[edged] (0,0.5) node[smalldot] {} -- (0,-0.7);}
\DeclareSymbol{2d2d}{0}{\draw[edged] (-0.7,2) node[dot] {} -- (0,0.5) -- (0.7,2) node[dot] {}; \draw[edged] (0,0.5) node[smalldot] {} -- (0,-0.7) node[smalldot] {} -- (1,0.6) node[dot] {};}
\DeclareSymbol{1d2d}{0}{\draw[edged] (0,0.7) -- (1.1,2) node[dot] {}; \draw[edged] (0,0.7) node[smalldot] {} -- (0,-0.7) node[smalldot] {} -- (1.1,0.4) node[dot] {};}
\DeclareSymbol{2d2d1}{0}{\draw[edged] (-0.7,2) node[dot] {} -- (0,0.5) -- (0.7,2) node[dot] {}; \draw[edged] (0,0.5) node[smalldot] {} -- (0,-0.7) -- (1,0.6) node[dot] {}; \draw[edge] (0,-0.7) node[smalldot] {} -- (0,-1.9);}
\DeclareSymbol{2d2d1d}{0}{\draw[edged] (-0.7,2) node[dot] {} -- (0,0.5) -- (0.7,2) node[dot] {}; \draw[edged] (0,0.5) node[smalldot] {} -- (0,-0.7) -- (1,0.6) node[dot] {}; \draw[edged] (0,-0.7) node[smalldot] {} -- (0,-1.9);}
\DeclareSymbol{tree1}{0}{\draw[edged] (-0.7,2) node[dot] {} -- (0,0.5) -- (0.7,2) node[dot] {}; \draw[edged] (0,0.5) node[smalldot] {} -- (0.5,-0.7) -- (1.3,0.6) node[dot] {}; \draw[edged] (0.5,-0.7) node[smalldot] {} -- (1,-1.9) node[smalldot] {} -- (1.8,-0.6) node[dot] {};}
\DeclareSymbol{tree2}{0}{\draw[edged] (-0.7,2) node[dot] {} -- (0,0.5)  -- (0.7,2) node[dot] {}; \draw[edged] (0,0.5) node[smalldot] {} -- (1.15,-0.7) node[smalldot] {} -- (2.3,0.5) node[smalldot] {}; \draw[edged] (1.6,2) node[dot] {} -- (2.3,0.5) -- (3,2) node[dot] {};}
\begin{document}

\title{The strong Feller property of the open KPZ equation}
\author{Alisa Knizel$^1$ and Konstantin Matetski$^2$}
\institute{University of Chicago, \email{aknizel@uchicago.edu}
\and Michigan State University, \email{matetski@msu.edu}}
\titleindent=0.65cm

\maketitle
\thispagestyle{empty}

\begin{abstract}
We prove that the semigroup, generated by the open KPZ equation on a
bounded spatial interval with Neumann boundary conditions
parametrized by real parameters $u$ and $v$, enjoys
the strong Feller property. From this, we conclude that for the values of the parameters satisfying either $u+v = 0$ or $u+v>0$, $\min (u,v)>-1$ the WASEP-stationary measure
constructed in \cite{CorwinKnizel} is the unique stationary measure
for the equation. It is expected that the same conclusion holds for
all values of $u$ and $v$.
\end{abstract}


\section{Introduction}

We consider $h : \R_+ \times [0,1] \to \R$ to be the solution of the \emph{open KPZ equation}
\begin{equ}[eq:openKPZ]
\partial_t h = \frac{1}{2} \partial^2_x h + \frac{1}{2} (\partial_x h)^2 + \xi, \qquad h(0,\cdot) = h_0(\cdot),
\end{equ}
with Neumann boundary conditions
\begin{equ}[eq:boundary]
\partial_x h(t,x) \Big|_{x = 0}= u, \qquad \partial_x h(t,x) \Big|_{x = 1}= -v,
\end{equ}
driven by a Gaussian space-time white noise $\xi$ on $\R \times
[0,1]$. The boundary parameters $u, v \in \R$ and a sufficiently
regular initial state $h_0 : [0,1] \to \R$ are fixed.

In contrast to the standard KPZ equation \cite{IntroToKPZ, HairerKPZ}, the difficulty with solving equation \eqref{eq:openKPZ} comes not only from the singular non-linearity, but also from the boundary conditions \eqref{eq:boundary}. Indeed, the standard analysis of the linearized version of \eqref{eq:openKPZ} suggests that the solution $h(t,x)$, provided it is defined, should have the same spatial regularity as a Brownian motion \cite{DaPratoZabczyk}, i.e., $h(t, \cdot) \in \CC^\alpha([0,1])$ for every $\alpha < \frac{1}{2}$. This particularly implies that the non-linearity in \eqref{eq:openKPZ} is a priori undefined. On the other hand, the boundary conditions \eqref{eq:boundary} are undefined either. The extension of the theory of regularity structures to domains with boundaries \cite{MateBoundary} allows to give a notion of solution to the open KPZ as a limit of \eqref{eq:openKPZ} with a mollified noise, after suitable renormalization. One of the interesting observations in \cite{MateBoundary} was that both the non-linearity and the boundary values should be renormalized. We provide more details on this solution in Section~\ref{sec:solution}.

Another way to solve equation \eqref{eq:openKPZ} is given by the standard Hopf-Cole transformation \cite{CorwinShen}, which yields the same solution as \cite{MateBoundary}. Namely, if we formally define  
\begin{equ}
Z(t,x) := e^{h(t,x)},
\end{equ} 
then $Z$ solves the \emph{stochastic heat equation} with a multiplicative noise
\begin{equ}[eq:SHE]
\partial_t Z = \frac{1}{2} \partial^2_x Z + Z \xi, \qquad Z(0,\cdot) = Z_0(\cdot),
\end{equ}
 with $Z_0(x) = e^{h_0(x)}$ and with inhomogeneous Robin boundary conditions
\begin{equ}
\partial_x Z(t,x) \Big|_{x = 0}= \Bigl(u - \frac{1}{2}\Bigr) Z(t,0), \qquad \partial_x Z(t,x) \Big|_{x = 1}= -\Bigl(v - \frac{1}{2}\Bigr) Z(t,1).
\end{equ}
In the case when the noise $\xi$ is smooth, this derivation is rigorous. If $\xi$ is a white noise, then \eqref{eq:SHE} can be solved classically, using a fixed point argument for the mild reformulation of the equation 
\begin{equ}
Z(t,x) = \int_0^1 P_{t}(x, y) Z_0(y) d y + \int_0^1 \int_{0}^t P_{t-s}(x,y) Z(s, y) \xi(d s, d y),
\end{equ}
with the last integral defined in the It\^{o} sense. Here, $P_{t}(x, y)$ is the heat kernel on $x, y \in [0, 1]$, which is the unique solution of equation
\begin{equ}
\partial_t P_{t}(x, y) = \frac{1}{2} \partial^2_x P_{t}(x, y), \qquad P_{0}(x, y) = \delta_{x - y},
\end{equ}
with boundary conditions 
\begin{equ}
\partial_x P_{t}(0, y) = \Bigl(u - \frac{1}{2}\Bigr) P_t(0,y), \qquad \partial_x P_{t}(1, y) = -\Bigl(v - \frac{1}{2}\Bigr) P_{t}(1, y),
\end{equ}
for all $t > 0$ and $y \in [0,1]$. More precisely, \cite[Prop.~2.7]{CorwinShen} states that if all moments of a random initial state $Z_0(x)$ are bounded uniformly in $x$, then there exists a unique adapted global solution of \eqref{eq:SHE} with a bounded second moment. If moreover $Z_0$ is positive, then for any $t > 0$ the solution $Z(t, x)$ is almost surely strictly positive. In this case the Hopf-Cole solution to the open KPZ equation \eqref{eq:openKPZ} is defined as $h(t,x) = \log Z(t,x)$.

If we set $u(t,x) = \partial_x h(t,x)$, where the differentiation is in the sense of distributions, then $u$ solves the stochastic Burgers equation 
\begin{equ}[eq:Burgers]
\partial_t u = \frac{1}{2} \partial^2_x u + \frac{1}{2} u^2 + \partial_x\xi, \qquad u(0,\cdot) = u_0(\cdot),
\end{equ}
with $u_0 = \partial_x h_0$ and with Dirichlet boundary conditions 
\begin{equ}
u(t,0) = u, \qquad u(t,1) = -v.
\end{equ}

As for the standard KPZ equation, the usual definition of a stationary measure does not make sense for \eqref{eq:openKPZ}, because of a time-dependent shift at the origin. This is however not the case for the stochastic Burgers equation \eqref{eq:Burgers}, in which the spatial derivative removes the shift. More precisely, a stationary measure of the open KPZ equation \eqref{eq:openKPZ} was defined in \cite{CorwinKnizel} as the law of a random function $\ch_{u, v} : [0, 1] \to \R$, with $\ch_{u, v}(0)= 0$, such that if $h$ is the solution of \eqref{eq:openKPZ} with the initial state $h_0 = \ch_{u, v}$, then for each $t > 0$ the law of the random function $x \mapsto h(t,x) - h(t, 0)$ is independent of $t$ and is equal to the law of $\ch_{u, v}$. Existence of a stationary measure for any $u, v \in \R$ was proved in \cite{CorwinKnizel}, and in the case $u + v \geq 0$ the distribution of $\ch_{u, v}$ was characterized by an exact formula for its multi-point Laplace transform. We provide more information about the stationary measure in Section~\ref{sec:measure}.

We see that the right space for the function $\ch_{u, v}$, whose law is defining a stationary measure, is $\CC^\alpha([0,1]) / \{f : f(x) = f(y)\, \forall x, y \in [0,1]\}$ for $\alpha \in (0, \frac{1}{2})$, i.e., the space of H\"{o}lder continuous functions quotient by constant shifts. This space is isomorphic to $\widetilde{\CC}^\alpha([0,1])$, which contains the functions $f \in \CC^\alpha([0,1])$ satisfying $f(0) = 0$. Hence, we will consider stationary measures as the laws of random functions $\ch_{u, v} \in \widetilde{\CC}^\alpha([0,1])$.

The following is the main result of this article, in which we use the strong Feller property, described in detail in Section~\ref{sec:Markov-operators}.

\begin{theorem}\label{thm:main}
For any boundary parameters $u, v \in \R$, the solution to the open
KPZ equation \eqref{eq:openKPZ} on $[0,1]$ is a Markov process on
$\CC^\alpha([0,1])$ for any $\alpha \in (0, \frac{1}{2})$, whose
semigroup enjoys the strong Feller property. Moreover, if either $u + v = 0$ or $u + v > 0$, $\min(u, v) > -1$, then this process has a unique WASEP-stationary measure $\ch_{u,v}$ which is defined in \cite{CorwinKnizel}.
\end{theorem}

\begin{remark}
From uniqueness of the stationary measure and \cite[Thm.~11.11]{DaPratoZabczyk} we conclude that if the parameters $u$ and $v$ are as in the statement of Theorem~\ref{thm:main}, then the law of $\ch_{u, v}$ is ergodic, i.e., the solution $h$ of equation \eqref{eq:openKPZ} satisfies
\begin{equ}
\lim_{T \to \infty} \frac{1}{T} \int_0^T F (h_t - h_t(0)) dt = \E[F(\ch_{u, v})]
\end{equ}
for any $F : \widetilde{\CC}^\alpha([0,1]) \to \R$ such that $\E[F(\ch_{u, v})^2] < \infty$.
\end{remark}

\subsection{Discussion}

The multipliers $\frac{1}{2}$ in \eqref{eq:openKPZ} are needed for
consistency with the previous work \cite{CorwinShen}, and the Feller
property will hold if one replaces them by any constants such that
the constant multiplier of $\partial^2_x h$ is strictly
positive.

The Laplace transform formulas for $\ch_{u,v}$ discovered in
\cite{CorwinKnizel} were inverted: in the mathematics literature it
was done in \cite{BKWW}, while in the physics literature it came in \cite{BLD}. These inversions provide a satisfying probabilistic description for the stationary
measures. Two different representations of the stationary measure were obtained in these papers.

In \cite{BLD} it was shown that $\ch_{u,v}$ is equal to the distribution of
$ W + X$ where $W$ and $X$ are independent stochastic
processes for $u, v>0$. More precisely, the process $W\in \CC([0,1])$ is a Brownian motion of variance $\frac{1}{2}$ and the law of $X\in \CC([0,1])$ is absolutely continuous with respect
to the law of a Brownian motion with variance $\frac{1}{2}$. In the case $u+v=0$, the process $X$ reduces to a Brownian motion with drift and one
sees that $ W + X$ has the law of standard Brownian motion
with drift $u=-v$ (i.e., the law of the random function on $[0,1]$
given by $y\mapsto B(y)+u y$ where $B$ is a Brownian motion with
variance $1$). In \cite{BLD} it was also conjectured that this representation holds for all values of $u$ and
$v$. From
\cite[Prop.~1.4]{BKWW} and \cite[Thm.~1.1]{BK} it follows that this
representation holds for  $u+v\geq 0$ and
$\min(u,v)>-1.$ 

We use these results in our work and it explains where
the restrictions on  $u, v$ in the statement of Theorem \ref{thm:main}
comes from. If this description of the stationary measure is valid for
all values of $u, v$ then our arguments will imply its uniqueness.

Uniqueness of the invariant measure also follows from the One Force -- One Solution principle, which was proved in \cite{Rosati} for the standard KPZ equation and was conjectured in \cite{CorwinKnizel} for the open KPZ equation. 

\begin{conjecture}\label{conj_unique}
Let us fix any $u,v\in \R$ and let $h_0,\tilde h_0\in \CC([0,1])$ be any random functions on the same probability space. Let furthermore $\xi$ be a white noise on $\R \times [0,1]$ on this probability space. Let $h(t,x) = h(t,x;t_0)$ and $\tilde h(t,x) = \tilde h(t,x;t_0)$ denote the solutions to the open KPZ equation \eqref{eq:openKPZ} driven by $\xi$ and started at time $t_0 \leq 0$ at $h_0$ and $\tilde h_0$ respectively. Then the following {\bf One Force -- One Solution principle} holds: for any real $s<s'$ the random functions $(t,x)\mapsto h(t,x)-h(t,0)$ and $(t,x)\mapsto \tilde h(t,x)-\tilde h(t,0)$ converge almost surely in $\CC([s,s'];\widetilde{\CC}([0,1]))$ to the same limit as $t_0\to -\infty$. 
\end{conjecture}

\subsection{Notation}

We use $\N$ for the set $\{0, 1, 2, \ldots\}$ of non-negative integer numbers. When working on the time-space domain $\R^2$ we use the parabolic scaling $\s = (2, 1)$, so that the distance between two points $z_i = (t_i,x_i) \in \R^2$ is given by
\begin{equ}
\| z_1 - z_2 \|_\s := |t_1 - t_2|^{1/2} + |x_1 - x_2|,
\end{equ}
where $t_i$ and $x_i$ are the time and space variables respectively. We denote by $\star$ the convolution on $\R^2$. 

With a little ambiguity we denote by $\CC^\alpha(\R)$, for $\alpha \in (0,1)$, the closure of smooth functions with respect to the H\"{o}lder norm
\begin{equ}
\| f \|_{\CC^\alpha} := \sup_{x \in \R} |f(x)| + \sum_{x \neq y} \frac{|f(x) - f(y)|}{|x-y|^\alpha}.
\end{equ}

The set $\CB$ contains all functions $\phi \in \CC^2(\R^2)$, supported in the ball of unit radius (with respect to the scaling $\s$), centered at the origin, and satisfying $\| \phi \|_{\CC^2} \leq 1$. For a function $\phi \in \CB$ we will use its rescaling and reentering 
\begin{equ}[eq:rescaled-phi]
\phi_{(t,x)}^\lambda(s, y) := \frac{1}{\lambda^{3}} \phi \left(\frac{s - t}{\lambda^2}, \frac{y - x}{\lambda}\right),
\end{equ}
for the scaling parameter $\lambda \in (0, 1]$.

The space $\CS'(\R^2)$ contains all Schwartz distributions. For $\alpha \in (2, 0)$ we define the space of Besov distributions from $\CC^\alpha$ to contain those distributions $\CS'(\R^2)$ which are obtained by the closure of smooth functions $\zeta: \R^2 \to \R$ with respect to the following system of seminorms: for any compact set $\fK \subset \R^2$ one has
\begin{equ}[eq:Besov]
\sup_{z \in \fK}|\zeta(\phi_{z}^\lambda)| \leq C \lambda^\alpha,
\end{equ}
uniformly over $\lambda \in (0,1]$ and $\phi \in \CB$, and for a
constant $C$ independent of $\lambda$ and $\phi$, but probably
dependent on $\fK$. Here, we write $\zeta(\phi_{z}^\lambda)$ for the
duality pairing $\langle \zeta, \phi_{z}^\lambda \rangle$. We denote
by $\| \zeta \|_{\alpha; \fK}$ the smallest value of the constant $C$
in \eqref{eq:Besov}. We will often use the set $\fK = [-2, T] \times
[0,1]$ with some $T > 0$, for which we simply write $\| \zeta
\|_{\alpha; T} := \| \zeta \|_{\alpha; \fK}$. The advantage to define
the Besov space as a close is in the separability of the space
$\CC^\alpha$. This property is important when using results from
Section~\ref{sec:Markov-operators}.
\subsection{Structure of the paper}

In Section~\ref{sec:measure} we describe the stationary measure for
the open KPZ equation, constructed in \cite{CorwinKnizel}, as well as
its representations in
subsequent works. The solution to the open KPZ equation in the
framework of regularity structures is provided in
Section~\ref{sec:solution}, and in Section~\ref{sec:Feller-regularity}
we use the results of \cite{HairerMattingly} to show that the Markov
semigroup of this solution enjoys the strong Feller property. The main
result of this article, Theorem~\ref{thm:main}, is proved in
Section~\ref{sec:proof-main}.

\subsubsection*{Acknowledgements}
A.~Knizel was partially supported by NSF grant DMS-2153958. K.~Matetski was partially supported by NSF grant DMS-19538.
The authors wish to thank Ivan Corwin for numerous fruitful discussions.

\section{A stationary measure}
\label{sec:measure}

Existence of a stationary measure for the open KPZ equation
\eqref{eq:openKPZ} was proved in \cite{CorwinKnizel}, for any boundary
parameters $u, v \in \R$. It was done by proving tightness of the
stationary measures of a properly rescaled open weakly asymmetric
simple exclusion process, and combining it with the convergence of the
latter to the solution of \eqref{eq:openKPZ}, proved in
\cite{CorwinShen, Parekh}. The following result contains only those parts of \cite[Thm.~1.4]{CorwinKnizel}, which are relevant for this article. 

\begin{theorem}\label{thm:WASEP-measure}
For any boundary parameters $u, v \in \R$, there is a random function $\ch_{u, v} \in \widetilde{\CC}^\alpha([0,1])$, for any $\alpha < \frac{1}{2}$, whose distribution is stationary for the open KPZ equation \eqref{eq:openKPZ}. Moreover, in particular cases it can be characterized as follows:
\begin{enumerate}\setlength\itemsep{0em}
\item In the case $u + v = 0$ one has $\ch_{u, v}(x) = B(x) + u x$, where $B$ is a standard Brownian motion.
\item In the case $u + v > 0$ we define 
\begin{equ}
\mathsf{C}_{u, v} := \begin{cases} 2 &\text{if}~~u \leq 0 \text{~~or~~} u \geq 1\,, \\ 2 u &\text{if}~~u \in (0,1)\,. \end{cases}
\end{equ}
Then the distribution of $\ch_{u, v}$ is uniquely characterized by the following exact formula for its multi-point Laplace transform:
\begin{equ}
\E \biggl[\exp\biggl(-\sum_{k=1}^d c_k \ch_{u, v}(x_k)\biggr)\biggr] = \frac{\E \Bigl[\exp\Bigl(\frac{1}{4} \sum_{k=1}^{d+1} (s_k^2 - \T_{s_k}) (x_k - x_{k-1})\Bigr)\Bigr]}{\E[-\T_0 / 4]},
\end{equ}
for any $d \in \N$, any $0 = x_0 < x_1 < \cdots < x_d \leq x_{d+1} = 1$, any $c_1, \ldots, c_d > 0$ such that $c_1 + \cdots + c_d < \mathsf{C}_{u, v}$, and for the values $s_k = c_k + \cdots + c_{d}$ and $s_{d+1} = 0$. Here, $\T_s$ is {\bf the continuous dual Hahn process}, defined in \cite{CorwinKnizel}.
\end{enumerate}
\end{theorem}

We will use the following representation of the stationary measure
obtained in
\cite{BKWW, BK, BLD}.

\begin{theorem}
Under the assumption $u + v>0$ and $\textup{min}(u, v) >-1$ the stationary measure
of the open KPZ equation \eqref{eq:openKPZ} on $[0,1]$ can be represented
in the form
\begin{equ}[eq:invariant-representation]
(\ch_{u, v}(x))_{x\in[0, 1]} \stackrel{\text{\tiny law}}{=} \bigl(W(x)+Y(x)-Y(0)\bigr)_{x\in [0, 1]},
\end{equ}
where $W$ is a Brownian motion of variance $\frac{1}{2}$ and
$(Y(x))_{x\in [0,1]}$
is an
$\R$-valued Markov process, such that the Radon-Nikodym derivative of the law of $(Y(x)-Y(0))_{x\in [0, 1]}$ with respect to
the law of the Brownian motion $W$ is given by
\begin{equ}[eq:RN-deriv]
\dfrac{d \P_Y}{d \P_W} (\beta)=\frac{1}{\CZ} e^{-2 v\beta(1)}\biggl(\int\limits_0^1e^{-2 \beta(x)} d x\biggr)^{-(u + v)}
  \end{equ}
with the argument of the density function denoted by
$\beta \in \CC([0,1])$ and an explicit positive normalization constant
$\CZ$.
\end{theorem}

The proof of this result in \cite{BKWW, BK}
  relies on the Laplace transform formula from \cite{CorwinKnizel} and uses the
  computations from \cite{BLD}, where this representation was first
  suggested. 
  
  Formula \eqref{eq:RN-deriv} in particular imply
  that the stationary distribution of the open KPZ equation is
  absolutely continuous with respect to the law of the Brownian
  motion. It is conjectured in \cite{BLD} that \eqref{eq:RN-deriv} is
  valued for all values of $u,v.$

\section{Properties of Markov operators}
\label{sec:Markov-operators}

In this section we collect several properties of Markov operators which guarantee uniqueness of invariant measures. More information on the topic can be found in \cite{Sinai, DPZ-Ergodicity}.

Let $\CU$ be a separable Banach space (in particular, it is Polish) and let $\fP$ be a Markov operator on $\CU$. It means that $\fP$ maps bounded measurable functions on $\CU$ to bounded measurable functions.

\begin{definition}
$\fP$ enjoys the \emph{strong Feller property} if for any bounded measurable function $\Psi : \CU \to \R$, the function $\fP \Psi : \CU \to \R$ is bounded and continuous with respect to the supremum norm.
\end{definition}

It $\fP$ enjoys the strong Feller property we will simply say that $\fP$ is a strong Feller operator. We are interested in \emph{invariant} measures of $\fP$, i.e., such measures $\mu$ on $\CU$ which satisfy 
\begin{equ}
\int_\CU (\fP \Psi)(u) \mu (d u) = \int_\CU \Psi(u) \mu (d u),
\end{equ}
for any bounded measurable function $\Psi : \CU \to \R$. A special role is played by ergodic invariant measures, which are roughly speaking those measures for which the law of large numbers hold (see Birkhoff-Khinchin ergodic theorem \cite{Sinai}).

\begin{definition}
An invariant probability measure $\mu$ for $\fP$ is \emph{ergodic} if for every set $A \subset \CU$, satisfying $(\fP \1_A)(u) = 1$ for all $u \in A$, one has $\mu(A) \in \{0, 1\}$. Here, we use the indicator function $\1_A(u) = 1$ if $u \in A$ and $\1_A(u) = 0$ otherwise.
\end{definition}

We recall that the \emph{topological support} $\supp(\mu)$ of $\mu$
consists of those points $u \in \CU$ such that $\mu(A) > 0$ for every
open neighborhood $A \subset \CU$ of $u$. The following result can be found in \cite{DPZ-Ergodicity}.

\begin{proposition}
Let $\fP$ be strong Feller and let $\mu \neq \nu$ be two ergodic invariant probability measures for $\fP$. Then $\supp(\mu) \cap \supp(\nu) = \emptyset$.
\end{proposition}

This result is usually stated for any mutually singular invariant probability measures $\mu$, $\nu$. Ergodicity and $\mu \neq \nu$ guarantee mutual singularity (see Theorem~2 in \cite[Ch.~I.1.2]{Sinai}).

The next result follows immediately from the preceding proposition, which can be also found in \cite[Cor.~3.9]{HairerMattingly}.

\begin{corollary}\label{cor:uniqueness}
Let $\fP$ be strong Feller with an invariant measure $\mu$, satisfying $\supp(\mu) = \CU$. Then $\mu$ is the only invariant measure. 
\end{corollary}

This corollary is the main result which we are going to use to prove uniqueness of the invariant measure of the open KPZ equation. Namely, we will prove that the solution is a Markov process whose transition semigroup enjoys the strong Feller property. Knowing then from \eqref{eq:RN-deriv} that an invariant measure is absolutely continuous with respect to the Brownian motion, we conclude that it has full support, and hence it is unique. 

\subsection{Comments on proving the strong Feller property}
\label{sec:how-to-prove}

In order to prove Theorem~\ref{thm:main}, we need to prove that for any $t > 0$, the transition semigroup $\fP_t$ of the solution to the open KPZ equation \eqref{eq:openKPZ} at time $t$ is strong Feller on a suitable Banach space $\CU$. We will do it by describing the evolution of the solution by a family of maps
\begin{equ}
\Phi_{s, t} : \CU \times \CM \to \CU, \qquad 0 \leq s \leq t \leq 1,
\end{equ}
so that given a solution $h$ at time $s$, the map $\Phi_{s, t}(h, \bxi)$ gives the solution at time $t$. Here, $\bxi$ is a noise taking values in a complete separable metric space $\CM$. In our case $\bxi$ will be a model (in the sense of the theory of regularity structures) constructed from the driving noise $\xi$ in \eqref{eq:openKPZ}. Of course, the maps $\Phi$ should be compatible, allowing us to iterate the solution, i.e., they should satisfy 
\begin{equ}
\Phi_{r, t} \bigl(\Phi_{s, r}(h, \bxi), \bxi\bigr) = \Phi_{s, t}(h, \bxi),
\end{equ}
for any $r \in [s,t]$ and any $(h, \bxi) \in \CU \times \CM$. 

As proved in \cite[Thm.~3.2]{HairerMattingly}, the two key results for $\Phi$, from which the strong Feller property of $\fP$ follows, are:
\begin{enumerate}\setlength\itemsep{0em}
\item For each $\bxi \in \CM$, the map $h \mapsto \Phi_{s, t}(h, \bxi)$ is Fr\'{e}chet differentiable on $\CU$.
\item Shifts of the initial data $h$ can be written in terms of shifts of the noise $\bxi$, and the Fr\'{e}chet derivatives of $\Phi$ with respect to these two shifts are linearly dependent.
\end{enumerate}
Then a perturbation of the initial data can be written in terms of a perturbation of the noise, while the latter can be analyzed using the Girsanov theorem. 

We refer to \cite[Sec.~2]{HairerMattingly} for a precise assumptions on the maps $\Phi$, which in particular allow their blow-ups in finite times. In Section~\ref{sec:Feller-regularity} we provide particular cases of these assumptions, which are sufficient to prove the strong Feller property for the transition semigroup of the solution of the open KPZ equation. However, in order to check these assumptions, we need to solve equation \eqref{eq:openKPZ} in the framework of regularity structures. This is what we do in the following section, in which we recall some of the results from \cite{MateBoundary}.

\section{The solution to the open KPZ equation}
\label{sec:solution}

We recall in this section how the solution to the open KPZ equation \eqref{eq:openKPZ} was constructed in \cite[Sec.~6.3]{MateBoundary}. We prefer to provide enough details, in order to be able to check transparently the assumptions stated in Section~\ref{sec:Feller-regularity}, which guarantee that the transition semigroup of the solution enjoys the strong Feller property. In order to keep the exposition more easy, we prefer to use the minimal regularity structure, which contains only those elements that are required to solve the equation.

\subsection{A mild form of the open KPZ equation}

It follows from \cite[Thm.~1.7]{MateBoundary} that if we use the theory of regularity structures to solve equation \eqref{eq:openKPZ}, then after reconstruction the boundary values $u$ and $-v$ get shifter by the constant 
\begin{equ}[eq:a]
a := \int_{\R^2} (\bar \rho * \rho)(s, y) \Bigl( \frac{1}{2} - \frac{1}{2} \mathrm{Erf} \Bigl( \frac{|y|}{\sqrt{2|s|}}\Bigr) - 2|y| \CN(s, y) \Bigr) ds dy,
\end{equ}
where the function $\rho : \R^2 \to \R$ is used to mollify the noise, $\bar \rho(s, y) = \rho(-s, -y)$, $\CN$ is the heat kernel and $\mathrm{Erf}$ is the Gauss error function. This is why, in order to obtain the required equation \eqref{eq:openKPZ} after reconstruction, we need to consider the shifted boundary conditions 
\begin{equ}[eq:buondary-new]
\partial_x h(t,x) \Big|_{x = 0}= u + a, \qquad \partial_x h(t,x) \Big|_{x = 1}= -v - a.
\end{equ}
It will be advantageous to write the open KPZ equation in a mild form. This is however much easier to do in the case of zero boundary condition, because we can write an explicit formula for the Neumann heat kernel. We formally observe that if $h$ is the solution of \eqref{eq:openKPZ} with the boundary conditions \eqref{eq:buondary-new}, then the transformed function 
\begin{equ}[eq:from-h-to-tilde-h]
\tilde{h}(t,x) = h(t,x) + \frac{v + a}{2} x^2 - (u + a) x
\end{equ}
solves the SPDE
\begin{equ}[eq:openKPZ-new]
\partial_t \tilde{h} = \frac{1}{2} \partial^2_x \tilde{h} + \frac{1}{2} (\partial_x \tilde{h})^2 + a_1 \partial_x \tilde{h} + a_2 + \xi, \qquad \tilde{h}(0,\cdot) = \tilde{h}_0(\cdot),
\end{equ}
with the initial condition $\tilde{h}_0(x) = h_0(x) + \frac{v + a}{2} x^2 - (u + a) x$ and with Neumann boundary conditions
\begin{equ}[eq:boundary-new]
\partial_x \tilde{h}(t,x) \Big|_{x = 0} = \partial_x \tilde{h}(t,x) \Big|_{x = 1} = 0.
\end{equ}
The polynomials $a_1$ and $a_2$ in \eqref{eq:openKPZ-new} depend only on the spatial variable and are given by
\begin{equ}[eq:a-functions]
a_1(x) := 2 \bigl(- 2 (v + a) x + u + a\bigr), \qquad a_2(x) := \frac{1}{2} \bigl(\bigl(-2 (v + a) x + u + a\bigr)^2  -v - a\bigr).
\end{equ}
The derivation of equation \eqref{eq:openKPZ-new} is of course correct of the driving noise $\xi$ is smooth and the equation can be solved in the classical sense. Formally, equation \eqref{eq:openKPZ-new} can be written in the mild form 
\begin{equs}[eq:openKPZ-mild]
&\tilde{h}(t,x) = \int_0^1 \widetilde{P}_{t}(x, y) \tilde{h}_0(y) d y + \frac{1}{2} \int_0^1 \int_{0}^t \widetilde{P}_{t-s}(x,y) \bigl(\partial_y \tilde{h}(s, y)\bigr)^2 d s\, d y \\
&\quad + \int_0^1 \int_{0}^t \widetilde{P}_{t-s}(x,y) \bigl( a_1(y) \partial_y \tilde{h}(s, y) + a_2(y)\bigr) d s\, d y + \int_0^1 \int_{0}^t \widetilde{P}_{t-s}(x,y) \xi(d s, d y),
\end{equs}
where $\widetilde{P}_{t}: [0, 1]^2 \to \R$ is the \emph{Neumann heat kernel}, which is characterized as the unique solution of the PDE
\begin{equ}
\partial_t \widetilde{P}_{t}(x, y) = \frac{1}{2} \partial^2_x \widetilde{P}_{t}(x, y), \qquad \widetilde{P}_{0}(x, y) = \delta_{x - y},
\end{equ}
for $t > 0$, $x \in (0,1)$ and $y \in [0,1]$ with the boundary conditions $\partial_x \widetilde{P}_{t}(0, y) = \partial_x \widetilde{P}_{t}(1, y) = 0$ for all $t > 0$ and $y \in [0,1]$. Let $P_t(x) = (4 \pi t)^{-\frac{1}{2}} e^{- x^2 / (4 t)}$ be the heat kernel on $x \in \R$. Then one can readily check that the Neumann heat kernel is given by 
\begin{equ}
\widetilde{P}_{t}(x,y) = \sum_{m \in \Z} \Bigl( P_t(x + y + 2m ) + P_t(x - y + 2 m) \Bigr).
\end{equ}

In the rest of this article we are going to work with equation \eqref{eq:openKPZ-mild}. In particular, we will solve it using the framework of regularity structures and then define the solution of the initial equation \eqref{eq:openKPZ} by the identity \eqref{eq:from-h-to-tilde-h}. We will demonstrate that our solution coincides with the one defined in \cite{MateBoundary}, and hence it coincides with the Hopf-Cole solution defined in \cite{CorwinShen}.

\subsection{A regularity structure}
\label{sec:RS}

We consider the set $\CW_\poly$, containing the monomials $X^\ell = X_0^{\ell_0} X_1^{\ell_1}$ for all $\ell = (\ell_0, \ell_1) \in \N^2$ with assigned degrees $\deg X^\ell = 2 \ell_0 + \ell_1$. Here, the dummy variables $X_0$ and $X_1$ correspond to the time and space variables respectively, and we write $\one$ for the unit element $X^0$ of degree $\deg \one = 0$. Note that we use the parabolic scaling $\s = (2,1)$ of $\R^2$ to compute degrees of the monomials, where $2$ is the scaling of time, and $1$ is the scaling of space. We use this particular scaling, since it corresponds to the parabolic time-space scaling induced by the linear part of equation \eqref{eq:openKPZ-new}.

The regularity structure of the open KPZ equation is the same as for the usual KPZ equation, and it was constructed in \cite{FrizHairer}. Namely, we define the smallest collections $\CU$ and $\CV$ of all formal expression, such that $\CW_\poly \in \CU \cap \CV$, $\Xi \in \CV$, and such that the implications hold
\begin{equs}
\tau \in \CV \setminus \CW_\poly \quad &\Rightarrow \quad \CI(\tau) \in \CU, \label{eq:rule1}\\
\tau_1, \tau_2 \in \CU \setminus \{\one\},~ \{\tau_1, \tau_2\} \not\subset \CW_\poly \quad &\Rightarrow \quad \partial \tau_1 \partial \tau_2 \in \CV. \label{eq:rule2}
\end{equs}
The symbol $\Xi$ corresponds to the noise $\xi$ in \eqref{eq:openKPZ-mild}, the symbol $\CI$ corresponds to the convolution with the hear kernel, and $\partial$ corresponds to a spatial derivative. We make the assumption that the product in \eqref{eq:rule2} is commutative, associative, with the unit $\one$. We also postulate the standard differentiation rules for monomials $\partial \one = 0$ and $\partial X^{(\ell_0, \ell_1)} = \ell_1 X^{(\ell_0, \ell_1-1)}$ for $\ell_1 \geq 1$. And it will be convenient to postulate $\CI(X^\ell) = 0$ for all $\ell \in \N^2$. We note that $\CV$ contains all elements $\d \tau$ such that $\tau \in \CU$, because they appear in \eqref{eq:rule2} as $\partial X_1 \partial \tau$. Respectively, $\CV$ contains the elements $2 X_1 \d \tau$ and $3 X_1^2 \d \tau$ for $\tau \in \CU$, because they come in \eqref{eq:rule2} from the expressions $\partial X^2_1 \partial \tau$ and $\partial X^3_1 \partial \tau$ respectively.  Then the set $\CV$ contains the elements corresponding to those appearing on the right-hand side of \eqref{eq:openKPZ-new}, while $\CU$ contains those which appear after integration in \eqref{eq:openKPZ-mild}.

We define the set of basis elements $\fT := \CU \cup \CV$. To each element in $\fT$ we assign a degree by $\deg \Xi = -\frac{3}{2} - \kappa$ for a fixed constant $\kappa \in (0, \frac{1}{10})$, and then recursively by
\begin{equ}
\deg (\tau_1 \tau_2) = \deg \tau_1 + \deg \tau_2, \qquad \deg \CI(\tau) = \deg \tau + 2, \qquad \deg \partial \bar \tau = \deg \bar\tau - 1,
\end{equ}
where we consider only those elements which appear in the construction \eqref{eq:rule1}/\eqref{eq:rule2}. Recall that the degrees of the elements of $\CW_\poly$ have been defined above. We will use the set $\langle \fT \rangle$ containing all finite linear combinations of the element of $\fT$. It will be convenient to denote $\CI'(\tau) = \CI(\d \tau)$. Then we view $\tau \mapsto \CI(\tau)$, $\tau \mapsto \CI'(\tau)$ and $\bar \tau \mapsto \d \bar \tau$ as linear maps on respective subspaces of $\langle \fT \rangle$.

In order to solve equation \eqref{eq:openKPZ-new} on a suitable space, it is enough to use only a finite number of elements of $\fT$. More precisely, we set 
\begin{equ}
\CW := \Bigl\{\tau \in \CV : \deg \tau \leq \frac{1}{2} + \kappa\Bigr\} \cup \Bigl\{\tau \in \CU : \deg \tau < \frac{3}{2} + 2\kappa\Bigr\},
\end{equ}
and we define the \emph{model space} $\CT$ as the set of all linear combinations of the element of $\CW$. Let $\CA$ be the set of all values $\deg \tau$ for $\tau \in \CW$. For $\alpha \in \CA$, we set the space $\CT_\alpha$ to contain all linear combinations of the element of degree $\alpha$. Respectively, for $m \in \R$ we set $\CT_{< m} := \bigoplus_{\alpha < m} \CT_\alpha$, where $\alpha$ takes values in $\CA$. Analogously we define the space $\CT_{\leq m}$. We denote by $\CQ_{< m}$ and $\CQ_{\leq m}$ the projections from $\langle \fT \rangle$ to $\CT_{< m}$ and $\CT_{\leq m}$ respectively. Furthermore,  every element $f \in \langle\fT\rangle$ can be uniquely written as $f = \sum_{\tau \in \fT} f_\tau \tau$ for a finite collections of non-vanishing values $f_\tau \in \R$, and we define 
\begin{equation*}
    |f|_\alpha := \sum_{\tau \in \fT : |\tau| = \alpha} |f_\tau|,
\end{equation*}
postulating $|f|_\alpha=0$ if the sum runs over the empty set.

It is convenient to use a graphical representation of the elements of $\CW$. For this, we denote the element $\Xi$ by a node ``\,\<X>\,''. Application of the map $\CI$ to a non-polynomial element $\tau$ is depicted by an edge ``\,\<edge>\,'' going down from the diagram representing $\tau$. For example, the element $\CI(\Xi)$ corresponds to the diagram ``\,\<1>\,''. Respectively, an edge ``\,\<edged>\,'' corresponds to the element $\CI'$, e.g., $\CI'(\Xi)$ is depicted by ``\,\<1d>\,''. The product of elements $\tau_1$ and $\tau_2$ is represented by the diagram, obtained from the two diagrams of these symbols by drawing them from the same root. For example, $\<2d>$ is the diagram for $\CI'(\Xi)^2$. We write the symbols for the monomials as before. We list in Table~\ref{tab:symbols} all the elements in $\CW$ and their degrees. 

\begin{table}[h]
\centering
\begingroup
\setlength{\tabcolsep}{10pt} 
\renewcommand{\arraystretch}{1.4}
	\subfloat{
    	\begin{tabular}{cc}
		\hline
		\textbf{Element} & \textbf{Degree} \\
		\hline
		$\Xi$ & $-\frac{3}{2}-\kappa$ \\		
		$\one$ & $0$ \\		
		$X_1$    & $1$ \\
		$\<1>$              & $\frac{1}{2}-\kappa$ \\
		$\<1d>$              & $-\frac{1}{2}-\kappa$ \\
		$\<1d1>$              & $\frac{3}{2}-\kappa$ \\
		$\<1d2d>$              & $-2 \kappa$ \\
	\end{tabular} \hspace{0.3cm}}
	\subfloat{\hspace{0.3cm}
	\begin{tabular}{cc}
		\hline
		\textbf{Element} & \textbf{Degree} \\
		\hline
		$\<2d>$              & $-1-2\kappa$ \\
		$\<2d1>$              & $1-2\kappa$ \\
		$\<2d1d>$              & $-2\kappa$ \\
		$\<2d2d>$              & $-\frac{1}{2}-3\kappa$ \\
		$\<2d2d1>$              & $\frac{3}{2}-3\kappa$ \\
		$\<tree1>$              & $-4\kappa$ \\
		$\<tree2>$             & $-4\kappa$ \\[-0.2cm]
	\end{tabular}}
\endgroup
	\caption{The elements in $\CW$ and their degrees. \label{tab:symbols}}
\end{table}

\subsubsection{A structure group}

Now, we will define a structure group $\CG$ of linear transformations on $\CT$. For this, we introduce the set $\CW_+$, containing $X_1$, the elements of $\CW$ of the form $\CI(\tau)$ (there are three of them), and the elements $\CI'(\Psi)$ and $\CI'(\CI'(\Psi^2) \Psi)$ where we write for brevity $\Psi = \CI'(\Xi)$. Then we define $\CT_+$ to be the free commutative algebra generated by the elements of $\CW_+$.

We define a linear map $\Delta : \CT \to \CT \otimes \CT_+$ by
\begin{subequations}\label{eqs:coproduct}
\begin{equation}
\Delta \Xi = \Xi \otimes \one, \qquad \Delta \one = \one \otimes \one, \qquad \Delta X_1 = X_1 \otimes \one + \one \otimes X_1,
\end{equation} 
and then, denoting by $I$ the identity operator on $\CT_+$, we recursively define
\begin{equs}
\Delta \tau_1 \tau_2 &= (\Delta \tau_1) (\Delta \tau_2),\label{eq:Delta2} \\
\Delta \CI(\tau) &= (\CI \otimes I) \Delta \tau + \one \otimes \CI(\tau), \qquad \text{for}~ \tau \notin \{\Psi, \CI'(\Psi^2) \Psi\}, \label{eq:Delta3} \\
\Delta \CI(\tau) &= (\CI \otimes I) \Delta \tau + \one \otimes \CI (\tau) + \one \otimes X_1 \CI' (\tau) +  X_1 \otimes \CI' (\tau), \label{eq:Delta4}\\
&\qquad  \text{for}~ \tau \in \{\Psi, \CI'(\Psi^2) \Psi\}, 
\end{equs}
\end{subequations}
for the respective elements $\tau_i, \tau \in \CW$. In Table~\ref{tab:Delta} we write the action of $\Delta$ on all elements of $\CW$.

\begin{table}[h]
\centering
\begingroup
\setlength{\tabcolsep}{10pt} 
\renewcommand{\arraystretch}{1.2}
	\subfloat{
    	\begin{tabular}{cc} 
		\hline
		\textbf{Element} & \textbf{Image} \\
		\hline
		$\Xi$ & $\Xi \otimes \one$ \\
		$\one$ & $\one \otimes \one$ \\
		$X_1$ & $X_1 \otimes \one + \one \otimes X_1$ \\
		$\<1>$ & $\<1> \otimes \one + \one \otimes \<1>$  \\
		$\<1d>$ & $\<1d> \otimes \one$  \\
		$\<1d1>$ & $\<1d1> \otimes \one + \one \otimes \<1d1> + \one \otimes \<1d1d>\, X_1 + X_1 \otimes \<1d1d>$  \\
		$\<1d2d>$ & $\<1d2d> \otimes \one$  \\
		$\<2d>$ & $\<2d> \otimes \one$  \\
		$\<2d1>$              & $\<2d1> \otimes \one + \one \otimes \<2d1>$ \\
		$\<2d1d>$              & $\<2d1d> \otimes \one$ \\
		$\<2d2d>$              & $\<2d2d> \otimes \one$ \\[0.1cm]
		$\<2d2d1>$              & $\<2d2d1> \otimes \one + \one \otimes \<2d2d1> + \one \otimes \<2d2d1d> X_1 + X_1 \otimes \<2d2d1d>$ \\[0.2cm]
		$\<tree1>$              & $\<tree1> \otimes \one$ \\[0.3cm]
		$\<tree2>$             & $\<tree2> \otimes \one$ \\
	\end{tabular}}
\endgroup
	\caption{The action of the operator $\Delta$ \label{tab:Delta} on the elements of $\CW$.}
\end{table}

Having the operator $\Delta$ at hand, for any linear functional $f : \CT_+ \to \R$ we define the map $\Gamma_{\!f} : \CT \to \CT$ as 
\begin{equation}\label{eq:Gamma-general}
\Gamma_{\!f} \tau := (I \otimes f) \Delta \tau.
\end{equation}
Then the structure group $\CG$ is defined as $\CG := \{\Gamma_{\!f} : f \in \CG_+\}$, where $\CG_+$ contains all multiplicative linear functionals $f : \CT_+ \to \R$, i.e., satisfying $f(\tau \bar \tau) = f(\tau) f(\bar \tau)$ for $\tau, \bar \tau \in \CT_+$. Note that such functions necessarily satisfy $f(\one) = 1$. For example, if we denote the values of this function 
\begin{equs}
f(X_1) &= a, \qquad f(\<1>) = b, \qquad f(\<2d1>) = c, \qquad f(\<1d1>) = d, \\
 &\quad f(\<1d1d>) = g, \qquad f(\<2d2d1>) = h, \qquad f(\<2d2d1d>) = w,
\end{equs}
then the action of the map $\Gamma_{\!f}$ on the elements of $\CW$ is provided in Table~\ref{tab:linear_transformations}.

\begin{table}[h]
\centering
\begingroup
\setlength{\tabcolsep}{10pt} 
\renewcommand{\arraystretch}{1.2}
	\subfloat{
    	\begin{tabular}{cc} 
		\hline
		\textbf{Element} & \textbf{Image} \\
		\hline
		$\Xi$ & $\Xi$ \\
		$\one$ & $1$ \\
		$X_1$ & $X_1 + a \one$ \\
		$\<1>$ & $\<1> + b \one$  \\
		$\<1d>$ & $\<1d>$  \\
		$\<1d1>$ & $\<1d1> + (d + a g) \one + g X_1$ \\
		$\<1d2d>$ & $\<1d2d>$ \\[0.23cm]
	\end{tabular} \hspace{0.3cm}}
	\subfloat{\hspace{0.5cm}
	    	\begin{tabular}{cc} 
		\hline
		\textbf{Element} & \textbf{Image} \\
		\hline
		$\<2d>$ & $\<2d>$  \\
		$\<2d1>$              & $\<2d1> + c\one$ \\
		$\<2d1d>$              & $\<2d1d>$ \\
		$\<2d2d>$              & $\<2d2d>$ \\
		$\<2d2d1>$              & $\<2d2d1> + (h + a w) \one + w X_1$ \\
		$\<tree1>$              & $\<tree1>$ \\
		$\<tree2>$             & $\<tree2>$ \\
	\end{tabular}}
\endgroup
	\caption{The action of $\Gamma_{\!f}$ on the elements in $\CW$. \label{tab:linear_transformations}}
\end{table}

One can readily check that the structure group $\CG$ consists of all linear maps $\Gamma : \CT \to \CT$ such that 
\begin{itemize}\setlength\itemsep{0em}
\item $\Gamma \Xi = \Xi$, $\Gamma \one = \one$ and $\Gamma X_1 = X_1 - x \one$ for some $x \in \R$.
\item For every $\tau \in \CT_\alpha$ one has $\Gamma \tau - \tau \in \CT_{< \alpha}$.
\item For every $\tau_1, \tau_2 \in \CW$, such that $\tau_1 \tau_2 \in \CW$ one has $\Gamma (\tau_1 \tau_2) = (\Gamma \tau_1) (\Gamma \tau_2)$.
\item For every $\tau \in \CW$ such that $\CI(\tau) \in \CW$, one has $\Gamma \CI(\tau) - \CI (\Gamma \tau) \in \CT_\poly$; the same identity holds for the map $\CI'$.
\end{itemize}

 Thus, we have constructed a regularity structure $\ST = (\CA, \CT, \CG)$. The fixed point equation on the regularity structure $\ST$ is formulated in \cite[Sec.~6.3]{MateBoundary}. In order to use the solution map in the following section, we would like to state the fixed point equation here. 

\subsection{The Neumann heat kernel}
\label{sec:Neumann}

Equation \eqref{eq:openKPZ-new} is more convenient to work with, because we have the explicit formula for the Neumann heat kernel $\widetilde{P}$ in the mild equation \eqref{eq:openKPZ-mild}. In particular, we can use \cite[Ex.~4.15]{MateBoundary}
to write the kernel in the form $\widetilde{P} = K + Z$, where the singular, compactly supported part $K : \R^2 \times \R^2 \setminus \{(z, z) \in \R^2\} \to \R$ has the properties from \cite[Def.~2.11]{MateBoundary} with the singularity at the diagonal of order $\beta = 2$. In particular, it can be decomposed into a sum of smooth localized functions $K = \sum_{n \geq 0} K_n$, such that $K_n$ is supported in a $2^{-n}$-neighborhood of the diagonal of $\R^2 \times \R^2$ and $\| D_1^k D_2^\ell K_n\|_{L^\infty} \lesssim 2^{(|k+\ell|_\s + 1)n}$ for any multiindices $k$ and $\ell$. This is a standard decomposition of the heat kernel, which is used to prove the Schauder estimates, and it was used in \cite{Regularity} to prove estimates in the framework of regularity structures. 

Due to the Neumann boundary conditions, the function $Z$ is more complicated than in the case of the periodic boundary conditions \cite{Regularity}, where it was smooth and compactly supported. In this case it has singularities on the boundaries of the space-time domain. More precisely, the initial condition is given on the set $P_0 = \{(0, x) : x \in \R\}$ and the Neumann boundary conditions are given on $P_1 = \R \times \{0, 1\}$. Then the function $Z : (\R^2 \setminus P)^2 \to \R$ is undefined on the boundaries $P = P_0 \cup P_1$ and has the properties from \cite[Def.~4.14]{MateBoundary} with the singularity at the boundary of order $\beta = 2$. In particular, one can write $Z = \sum_{n \geq 0} Z_n$, where each function $Z_n$ is smooth, satisfies $\| D_1^k D_2^\ell Z_n\|_{L^\infty} \lesssim 2^{(|k+\ell|_\s + 1)n}$ and is supported on 
\begin{equ}
\bigl\{(z_1, z_2) = ((t_1,x_1), (t_2,x_2)) : |z_1|_{P_1} + |z_2|_{P_1} + |t_1 - t_2|^{1/2} \lesssim 2^{-n}\bigr\},
\end{equ}
where $|(t,x)|_{P_1} := |x| \wedge |x-1| \wedge 1$ is the distance to the boundary $P_1$.

\subsection{Models}
\label{sec:model}

The notion of model for the regularity structure is exactly the same as in \cite{Regularity} and is independent of the boundary conditions on the domain on which the SPDE is being solved. More precisely, a model $\CM = (\Pi, \Gamma)$ consists of the following elements: 
\begin{itemize}
\item A map $\Gamma : \R^2 \times \R^2 \to \CG$ such that $\Gamma_{z_1 z_2} \circ \Gamma_{z_2 z_3} = \Gamma_{z_1 z_3}$ for all points $z_i \in \R^2$.
\item A collection of continuous linear maps $\Pi_z : \CT \to \CS'(\R^2)$ such that $\Pi_{z_1} \circ \Gamma_{z_1 z_2} = \Pi_{z_2}$ for all $z_i \in \R^2$.
\end{itemize}
Furthermore, these maps are required so satisfy the following bounds: for every compact set $\fK \subset \R^2$ one has 
\begin{equ}[eq:model-bounds]
|(\Pi_z \tau) (\phi_z^\lambda)| \lesssim \lambda^{\deg \tau}, \qquad \| \Gamma_{z \bar z} \tau\|_{m} \lesssim \| z - \bar z\|_{\s}^{\deg \tau - m}
\end{equ}
uniformly in $z, \bar z \in \fK$, $\lambda \in (0,1]$, $\phi \in \CB$, $\tau \in \CW$, and $m < \deg \tau$. The proportionality constant in this bound can depend on $\fK$, but is independent of the other quantities. We always consider models compatible with polynomials in the sense $(\Pi_{(t,x)} \one)(s, y) = 1$ and $(\Pi_{(t,x)} X_1)(s, y) = y - x$.

We denote by $\| \Pi \|_{\fK}$ and $\| \Gamma\|_{\fK}$ the smallest constants such that the bounds \eqref{eq:model-bounds} hold. We will often work on the set $\fK = [-2, T] \times [0, 1]$ for $T > 0$, in which case we will simply write $\| \Pi \|_{T}$ and $\| \Gamma\|_{T}$. We will also use $\| \CM \|_T = \| \Pi \|_{T} + \| \Gamma\|_{T}$.

\subsubsection{A canonical lift of the noise}
\label{sec:lift}

Since the definition of the model is independent of the boundary conditions, this allows us to use the admissible model $\CM_{\KPZ} = (\Pi, \Gamma)$ on $\ST$ constructed for the periodic KPZ equation in \cite{FrizHairer}, using the singular part $K$ of the heat kernel (see Section~\ref{sec:Neumann}). The construction was done via smooth approximations. Namely, the authors considered a smooth mollification $\xi_\eps$ of the driving noise, such that $\xi_\eps \to \xi$ in $\CC^\alpha(\R^2)$ for any $\alpha < -\frac{3}{2}$. The smooth noise $\xi_\eps$ was lifted to a model $(\Pi^\eps, \Gamma^\eps)$, such that after a suitable renormalization the model was proved to have a limit $(\Pi, \Gamma)$ as $\eps \to 0$ with respect to the metric on the space of models. Since only the singular part $K$ of the heat kernel was used in this definition, and $K$ has the same properties for the Neumann and periodic heat kernels, the model $(\Pi, \Gamma)$ is indeed defined in the same way and is independent of the boundary conditions. Since we are not going to use the smooth approximations, we prefer not to go into more details and refer to \cite{FrizHairer}.

\subsection{Modeled distributions}

Using a model $(\Pi, \Gamma)$, the set of singular modeled distributions $\SD^{\gamma, w}_P$ is introduced in \cite[Def.~3.2]{MateBoundary}, for $\gamma > 0$ measuring the regularity of modeled distributions, and $w = (\eta, \sigma, \mu) \in \R^3$ measuring singularities on the boundaries $P$ (the latter is defined in Section~\ref{sec:Neumann}). More precisely, the values $\eta$, $\sigma$ and $\mu$ describe the orders of singularities at $P_0$, $P_1$ and $P_0 \cap P_1$ respectively. 

Before defining the space, we need to introduce several quantities. As before we use the distance to the boundary $|(t,x)|_{P_1} := |x| \wedge |x-1| \wedge 1$ is the distance to the boundary $P_1$. We also define the distance $|(t,x)|_{P_0} := |t|^{1/2} \wedge 1$ to the boundary $P_0$, and we set 
\begin{equ}
|z_1, z_2|_{P_i} := |z_1|_{P_i} \wedge |z_2|_{P_i}.
\end{equ}
For any compact set $\fK \subset \R^2$ we define 
\begin{equ}
\fK_P := \bigl\{ (z_1, z_2) \in (\fK \setminus P)^2 : z_1 \neq z_2,\; \|z_1 - z_2\|_\s \leq |z_1, z_2|_{P_0} \wedge |z_1, z_2|_{P_1} \bigr\}.
\end{equ}
The space $\SD^{\gamma, w}_P$ contains those $f : \R^2 \setminus P \to \CT_{< \gamma}$ such that for every compact set $\fK \subset \R^2$ the following norm is finite 
\begin{equs}[eq:norm-definition]
&\$ f \$_{\gamma, w; \fK} := \sup_{z \in \fK : 0 < |z|_{P_0} \leq |z|_{P_1}} \sup_{m < \gamma} \frac{\| f(z) \|_{m}}{|z|^{\mu - m}_{P_1} (|z|_{P_0} / |z|_{P_1})^{(\eta - m) \wedge 0}} \\
&\qquad + \sup_{z \in \fK : 0 < |z|_{P_1} \leq |z|_{P_0}} \sup_{m < \gamma} \frac{\| f(z) \|_{m}}{|z|^{\mu - m}_{P_0} (|z|_{P_1} / |z|_{P_0})^{(\sigma - m) \wedge 0}} \\
&+ \sup_{(z_1, z_2) \in \fK_P} \sup_{m < \gamma} \frac{\| f(z_1) - \Gamma_{z_1, z_2} f(z_2) \|_{m}}{\| z_1 - z_2\|_\s^{\gamma - m} |z_1, z_2|^{\eta - \gamma}_{P_0} |z_1, z_2|^{\sigma - \gamma}_{P_1} (|z_1, z_2|_{P_0} \vee |z_1, z_2|_{P_1})^{\mu - \eta - \sigma + \gamma}}.
\end{equs}
Properties of the functions from $\SD_P^{\gamma, w}$ can be found in Sections~3 and 4 in \cite{MateBoundary}. The space $\SD^{\gamma, w}_P$ depends on the map $\Gamma$ of the underlying model, and we will write $\SD^{\gamma, w}_P(\Gamma)$ when we want to show this dependence. We denote by $\| f \|_{\gamma, w; \fK}$ the sum of the first two terms in \eqref{eq:norm-definition}; we note that this quantity is independent of the underlying model. Given then another model $(\bar \Pi, \bar \Gamma)$ and a modeled distribution $\bar f \in \SD^{\gamma, w}_P(\bar \Gamma)$, we define the distance 
\begin{equs}
&\$ f; \bar f \$_{\gamma, w; \fK} := \| f - \bar f \|_{\gamma, w; \fK} \\
&\quad + \sup_{(z_1, z_2) \in \fK_P} \sup_{m < \gamma} \frac{\| f(z_1) - \bar f(z_1) - \Gamma_{z_1, z_2} f(z_2) + \bar \Gamma_{z_1, z_2} \bar f(z_2) \|_{m}}{\| z_1 - z_2\|_\s^{\gamma - m} |z_1, z_2|^{\eta - \gamma}_{P_0} |z_1, z_2|^{\sigma - \gamma}_{P_1} (|z_1, z_2|_{P_0} \vee |z_1, z_2|_{P_1})^{\mu - \eta - \sigma + \gamma}}.
\end{equs}
When we work on the set $\fK = [-2, T] \times [0, 1]$, we will simply write $\$ f \$_{\gamma, w; T}$ and $\$ f; \bar f \$_{\gamma, w; T}$.

The reconstruction map $\CR$, associated to a model $(\Pi, \Gamma)$, maps modeled distributions to actual functions or distributions on $\R^2$. We refer to \cite{Regularity, MateBoundary} for the definition of $\CR$, the proof of its existence and its action on the just introduced singular modeled distributions. 

\subsection{A fixed point problem}
\label{sec:fixed-point}

Using the definitions from the previous sections, we can now solve equation \eqref{eq:openKPZ-new} as a fixed point problem on a suitable space of singular modeled distributions. For this we need to make several definitions. Throughout this section we are going to work with the model $\CM_{\KPZ} = (\Pi, \Gamma)$ introduced in Section~\ref{sec:lift}.

The fixed point problem can be stated for the remainder of a one step iteration of \eqref{eq:openKPZ-new}. Namely, we define the processes 
\begin{equ}[eq:psi]
\psi(t,x) := \int_0^1 \int_{0}^t \widetilde{P}_{t-s}(x,y) \xi(d s, d y)
\end{equ}
and $w(t,x) := \tilde{h}(t, x) - \psi(t,x)$. Then from \eqref{eq:openKPZ-new} we conclude that $w$ solves equation (for smooth $\xi$ these computations are correct, while for a space-time white noise $\xi$ they are formal)
\begin{equ}[eq:openKPZ-DaPrato]
\partial_t w = \frac{1}{2} \partial^2_x w + \frac{1}{2} \bigl( (\partial_x w)^2 + 2 \partial_x w \partial_x \psi + (\partial_x \psi)^2\bigr) + a_1 \bigl(\partial_x w + \partial_x \psi\bigr) + a_2,
\end{equ}
with the initial condition $w(0,\cdot) = \tilde{h}_0(\cdot)$.

Let $\kappa > 0$ be as in the definition of the degree of the symbol $\Xi$ in Section~\ref{sec:RS}. Then we set $\gamma_0 = \kappa$. Using the decomposition of the Neumann heat kernel $\widetilde{P} = K + Z$ (see Section~\ref{sec:Neumann}), we can introduce the integration map 
\begin{equ}
\CP := \CK_{\gamma_0} + Z_{\gamma_0} \CR,
\end{equ}
where $\CR$ is the reconstruction map associated to the model, and $\CK_{\gamma_0}$ and $Z_{\gamma_0}$ are defined in Eq.~2.8 and Remark~4.17 in \cite{MateBoundary} respectively. Let us define 
\begin{equ}
\Psi := \CP (\bR^D_+ \Xi),
\end{equ}
with the projection $\bR^D_+ (t,x) = \one$ if $t > 0$, $x \in D := [0, 1]$, and $\bR^D_+ (t,x) = 0$ otherwise. From the properties of the involved operators one has $\CR \Psi = \psi$, where the latter is defined in \eqref{eq:psi}. Then the fixed point problem corresponding to equation \eqref{eq:openKPZ-DaPrato} is 
\begin{equ}[eq:fixed-pt]
W = \frac{1}{2}\CP \bR^D_+ \bigl( (\d W)^{2} + 2 (\d W) (\d \Psi) + (\d \Psi)^{2} \bigr) + \CP \bR^D_+ \bigl( a_1 \d W + a_1 \d \Psi + a_2 \bigr) + \widetilde{P} \widetilde{h}_0,
\end{equ}
where $\d$ is the abstract derivative on the regularity structure (see Section~\ref{sec:RS}), $\widetilde{P} \widetilde{h}_0$ is the polynomial lift of the harmonic extension, and the functions $a_1$ and $a_2$, defined in \eqref{eq:a-functions}, are naturally lifted to polynomials as 
\begin{equ}
a_1(X_1) := 2 \bigl(- 2 (v + a) X_1 + u + a\bigr) \one, \quad a_2(X_1) := \frac{1}{2} \bigl(\bigl(-2 (v + a) X_1 + u + a\bigr)^2  -v - a\bigr)\one.
\end{equ}

In order to solve equation \eqref{eq:fixed-pt}, we need to specify the respective spaces. The term $(\d \Psi)^{2}$ has the lowest degree on the right-hand side of \eqref{eq:fixed-pt}, $\deg (\d \Psi)^{2} = -1 - 2\kappa$. Then $W$ takes values in the sector $V := \CI(\CT_{\geq -1 - 2\kappa}) + \CT_{\poly}$ with the smallest degree $\alpha = 0$. The terms on the right-hand side of \eqref{eq:fixed-pt} take values in the sectors 
\begin{equs}[eq:RHS]
\frac{1}{2}(\d W)^{2} &\in V_0 := (\d V)^{\star 2}, \\
\frac{1}{2}(\d W) (\d \Psi) &\in V_1 := (\d V) \star \CT_{\geq -1/2 - \kappa}, \\
\frac{1}{2}(\d \Psi)^{2} &\in V_2 := \CT_{\geq -1 - 2\kappa}, \\
 a_1 \d W &\in V_3 := \d V, \\
  a_1 \d \Psi &\in V_4 := \CT_{\geq -1/2 - \kappa}, \\
   a_2 &\in V_5 := \CT_{\poly}.
\end{equs}
We denote by $\alpha_i$ the smallest degree of the sector $V_i$, so that $\alpha_0 = -4\kappa$, $\alpha_1 = -\frac{1}{2} - 3\kappa$, $\alpha_2 = -1 - 2\kappa$, $\alpha_3 = - 2\kappa$, $\alpha_4 = -\frac{1}{2} - \kappa$ and $\alpha_5 = 0$. Now, we need to specify particular spaces of modeled distributions, in which these terms live. Computations of the regularities and orders of singularities at the boundary follow from the properties of singular modeled distributions provided in \cite{MateBoundary}. We will be looking for a solution $W \in \SD_P^{\gamma, w}$ with 
\begin{equ}[eq:gamma]
\gamma = \frac{3}{2} + \kappa, \qquad \eta = \kappa, \qquad \sigma = \frac{1}{2} + 2 \kappa,  \qquad w = (\eta, \sigma, 0),
\end{equ}
such that $\d W \in \SD_P^{\gamma - 1, (\eta-1, \sigma-1, \kappa - 1)}$. We note that $\d W$ takes values in a sector with the minimal degree $\alpha_3 = -2 \kappa$. Then the $i$-th term in \eqref{eq:RHS} takes values in the space $\SD_P^{\gamma_i, w_i} (V_i)$ with $\gamma_i = \kappa$ and $w_i = (\eta_i, \sigma_i, \mu_i)$, where  
\begin{equs}
\eta_0 &= 2 (\eta - 1) = 2 \kappa - 2, \qquad &\sigma_0 &= 2 (\sigma - 1) = 4 \kappa - 1, \qquad &\mu_0 &= 2 \kappa - 2, \\
\eta_1 &= \eta - 1 + \alpha_4 = -\frac{3}{2}, \qquad &\sigma_1 &= \sigma - 1 + \alpha_4 = \kappa - 1, \qquad &\mu_1 &= -\frac{3}{2}, \\
\eta_2 & = -1 - 2\kappa, \qquad &\sigma_2 & = -1 - 2\kappa, \qquad &\mu_2 & = -1 - 2\kappa, \\
\eta_3 &= \eta - 1 = \kappa - 1, \qquad &\sigma_3 &= \sigma - 1 = -\frac{1}{2} + 2 \kappa, \qquad &\mu_3 &= \kappa - 1, \\
\eta_4 & = -\frac{1}{2} - \kappa, \qquad &\sigma_4 & = -\frac{1}{2} - \kappa, \qquad &\mu_4 & = -\frac{1}{2} - \kappa, \\
\eta_5 &= 0, \qquad &\sigma_5 &= 0, \qquad &\mu_5 & = 0.
\end{equs}

The value of $\sigma_2$ is smaller than $1$ and does not satisfy the assumptions of \cite[Lem.~4.12]{MateBoundary}. In particular, the abstract integration operator $\CP$ can not be applied to $(\d \Psi)^{2}$. In order to resolve this problem, the singularities at the boundary should be renormalised, what was done in \cite[Sec.~6.3]{MateBoundary}. Then Theorem~5.6 in \cite{MateBoundary} applies, which implies that equation \eqref{eq:fixed-pt} has a unique local solution $W \in \SD_P^{\gamma, w}$. However, the reconstructed solution $\tilde{h} (t,x) = \CR (W + \Psi)(t,x)$ solves equation \eqref{eq:openKPZ-new} with the Neumann boundary conditions 
\begin{equ}
\partial_x \tilde{h}(t,x) \Big|_{x = 0}= - a, \qquad \partial_x \tilde{h}(t,x) \Big|_{x = 1}= a,
\end{equ}
with the constant $a$ defined in \eqref{eq:a}. Performing the transformation \eqref{eq:from-h-to-tilde-h}, we conclude that $h$ solves the open KPZ equation \eqref{eq:openKPZ}. As follows from the argument provided at the end of Section~6.3 in \cite{MateBoundary}, this solution coincides with the Hopf-Cole solution defined in \cite{CorwinShen}.

\section{The strong Feller property of the open KPZ equation}
\label{sec:Feller-regularity}

To get the strong Feller property for the semigroup generated by the open KPZ equation, we will use the general result provided in \cite[Thm.~4.8]{HairerMattingly}. This result relies on several technical assumption, and we prefer not to duplicate them here but rather refer to \cite[Assums.~6-11]{HairerMattingly}. To check these assumptions, we will use the terminology of the theory of regularity structures, which can be found in \cite{Regularity}.

We denote by $\MM$ the space of all admissible models on $\ST$, i.e., those models which act naturally on monomials $X^\ell$ and on the abstract integration map $\CI$. We consider $\fR$ to be a subgroup of the renormalization group. We have a continuous ``canonical lift map'' $\LL : \CC^\infty(\R^2) \to \MM$ and a finite dimensional Lie group $\fR$, with a continuous action $M$ of $\fR$ onto $\MM$. As we stated in the beginning of Section~\ref{sec:model}, the canonical lift of the noise to a model does not depend on the boundary conditions and coincides with the canonical lift for the KPZ equation. The Lie group $\fR$ is again the same as for the KPZ equation \cite[Sec.~15.5]{FrizHairer}. More precisely, the group $\fR$ is $4$-dimensional, whose action $M_g$ (with $g \in \fR$) on the set of trees $\fT$ defined in Section~\ref{sec:RS} can be written in the form $M_g = \exp(- \sum_{i = 0}^3 C_i L_i)$ for fixed constants $C_0, \ldots, C_3$ and for the generators $L_i$ determined by the following contraction rules:
\begin{equ}
L_0\, \<1d2d> = \one, \qquad L_1\, \<2d> = \one, \qquad L_2\, \<tree2> = \one, \qquad L_3\, \<tree1> = \one,
\end{equ}
where we use the diagrams from Table~\ref{tab:symbols}. These definitions should be understood as follows: if $\tau$ is any diagram then $L_i \tau$ is the sum of all diagrams obtained from $\tau$ by performing the contraction rule for $L_i$. The image of $M_g$, applied to the elements of Table~\ref{tab:symbols}, is provided in Table~\ref{tab:renormalization}. The action $M_g$ can be extended to the space of admissible models $\CM$ as in \cite[Sec.~8]{Regularity}. This shows that the first part of Assumption~6 in \cite{HairerMattingly} holds.

\begin{table}[h]
\centering
\begingroup
\setlength{\tabcolsep}{10pt} 
\renewcommand{\arraystretch}{1.4}
	\subfloat{
    	\begin{tabular}{cc}
		\hline
		\textbf{Element $\tau$} & \textbf{Image $M_g \tau$} \\
		\hline
		$\Xi$ & $\Xi$ \\		
		$\one$ & $\one$ \\		
		$X_1$    & $X_1$ \\
		$\<1>$              & $\<1>$ \\
		$\<1d>$              & $\<1d>$ \\
		$\<1d1>$              & $\<1d1>$ \\
		$\<1d2d>$              & $\<1d2d> - C_0 \one$ \\
	\end{tabular} \hspace{0.3cm}}
	\subfloat{\hspace{0.3cm}
	\begin{tabular}{cc}
		\hline
		\textbf{Element $\tau$} & \textbf{Image $M_g \tau$} \\
		\hline
		$\<2d>$              & $\<2d> - C_1 \one$ \\
		$\<2d1>$              & $\<2d1>$ \\
		$\<2d1d>$              & $\<2d1d>$ \\
		$\<2d2d>$              & $\<2d2d> - 2 C_0\, \<1d>$ \\
		$\<2d2d1>$              & $\<2d2d1> - 2 C_0\, \<1d1>$ \\
		$\<tree1>$              & $\<tree1> - 2 C_0\, \<1d2d> - C_0 \<2d1d> - C_3 \one$ \\
		$\<tree2>$             & $\<tree2> - C_2 \one$ \\[-0.15cm]
	\end{tabular}}
\endgroup
	\caption{The image of $M_g$ applied to the diagrams listed in Table~\ref{tab:symbols}. Some trees (like the second tree in the second column) do not get renormalized because of our assumption $\CI'(\one) = \CI(\one) = 0$. \label{tab:renormalization}}
\end{table}

Our next goal is to check the assumption on non-linearities in SPDEs. For this we write equation \eqref{eq:fixed-pt} as 
\begin{equ}[eq:fixed-pt-new]
W = \sum_{i = 0}^5 \CP \bR^D_+ F_i(W, \Psi) + \widetilde{P} \widetilde{h}_0.
\end{equ}
As we describe in Section~\ref{sec:fixed-point}, each function $F_i : \CQ_{< \gamma}\CU \times \langle \Psi \rangle \to \CQ_{< \bar \gamma_i} V_i$ is strongly locally Lipschitz continuous from $\SD_P^{\gamma, w} \times \langle \Psi \rangle$ to $\SD_P^{\gamma_i, w_i}$ over the underlying admissible model. Fixing the constant $\zeta = 1 - 2 \kappa$, we have furthermore the inclusions 
\begin{equ}
\CT_\poly \subset \CU \subset \CT_\poly \oplus \CT_{\geq \zeta}, \qquad \CT_\poly \subset V_i \subset \CT_\poly \oplus \CT_{\geq \alpha_i},
\end{equ}
where the constants $\alpha_i$ are as in Section~\ref{sec:fixed-point}. This shows that the second part of Assumption~6 in \cite{HairerMattingly} holds.

The next assumption we are going to check is about driving noises of SPDEs and their lifts to admissible models. We need this assumption to hold only for the Gaussian noise which drives the open KPZ equation \eqref{eq:openKPZ}.

We have the following  from the analogous result for the periodic KPZ equation \cite{FrizHairer}. For every mollifier $\rho \in \CC^\infty_0(\R^2)$ there is a sequence $g_\eps \in \fR$, such that the sequence of random models $M_{g_\eps} \LL (\xi \star \rho^\eps)$ (we recall that $M$ is an action of the renormalisation group $\fR$ on models in $\CM$) converges in probability in $\CM$ to a limiting random model $\bxi = (\Pi, \Gamma)$. Moreover, for every $z = (t, x), \bar z = (\bar t, \bar x) \in [0, 1] \times \R$ such that $\bar t \leq t$, the random variables $(\Pi_z \tau)(\phi_z)$ and $\Gamma_{z \bar z}$ are $\CF_t$-measurable, for every $\phi \in \CC_0^\infty(\R^2)$ supported on $\R_- \times \R$. Measurability of the model follows from its explicit construction in \cite[Sec.~15.4]{FrizHairer}. This is Assumption~7 in \cite{HairerMattingly}.

Now, we are going to check Assumption~8 in \cite{HairerMattingly}. For $g \in \fR$ and a smooth function $\xi \in \CC^\infty(\R^2)$, let us consider the model of the form $\bxi_g = M_g \LL(\xi)$, where we recall that $\LL$ is a lift of $\xi$ to a model and $M_g$ is its renormalisation. Let $W$ be the solution to \eqref{eq:fixed-pt-new} (which is just another way to write \eqref{eq:fixed-pt}) and let us define its reconstruction $w = \CR_g W$. We need to derive a PDE solved by $w$. More precisely, we are going to show that $w$ solves
\begin{equ}[eq:renormalized-equation]
\partial_t w = \frac{1}{2} \partial^2_x w + \frac{1}{2} \bigl( (\partial_x w)^2 + 2 \partial_x w \partial_x \psi + (\partial_x \psi)^2 - c_g^{(1)} \partial_x w - c_g^{(2)} \partial_x \psi - c_g^{(3)}\bigr) + a_1 \bigl(\partial_x w + \partial_x \psi\bigr) + a_2,
\end{equ}
for some constants $c_g^{(i)}$, $i = 1, 2, 3$, depending on $g$, and where the functions $\psi$ and $a_i$ are the same as in \eqref{eq:openKPZ-DaPrato}.

It is more convenient to write equation \eqref{eq:fixed-pt} as 
\begin{equ}[eq:fixed-point-proof]
W = \frac{1}{2} \CI \bigl( (\d W)^{2} + 2 (\d W) (\d \Psi) + (\d \Psi)^{2} \bigr) + \CI \bigl( a_1 \d W + a_1 \d \Psi \bigr) + \texttt{Rem}(W),
\end{equ}
where the remainder $\texttt{Rem}(W)$ is an abstract polynomial. Performing Picard iterations in this equation, we can see that the solution $W$ can be written in the form 
\begin{equ}[eq:W-expansion]
W = w\, \one\,+ \tilde w\, X_1\, + \frac{1}{2}\, \<2d1>\, + \frac{1}{4}\, \<2d2d1>\,+ \Bigl(a_1(0) + \frac{\tilde w}{2}\Bigr)\, \<1d1>\,,
\end{equ}
for some real-valued functions $w$ and $\tilde w$. We note that the solution theory of the equation, provided in Section~\ref{sec:fixed-point}, suggests that $W \in \SD_P^{\gamma, w}$ where $\gamma$ is slightly larger than $\frac{3}{2}$. Hence, all the elements in the expansion \eqref{eq:W-expansion} with degrees greater than $\deg \<1d1>$ can be omitted. Applying the abstract derivative, we get 
\begin{equ}
\d W = \tilde w\, \one\, + \frac{1}{2}\, \<2d1d>\, + \frac{1}{4}\, \<2d2d1d>\,+ \Bigl(a_1(0) + \frac{\tilde w}{2}\Bigr)\, \<1d1d>\,,
\end{equ}
which is a modeled distribution of regularity $\gamma-1$ (which is slightly greater than $\frac{1}{2}$). The diagrams $\<2d2d1d>$ and $\<1d1d>$ are not included into the regularity structure $\ST$ (see Table~\ref{tab:symbols}), because they are not required to make sense of the non-linearity in equation \eqref{eq:fixed-point-proof}. These diagrams are however elements of the set of trees $\fT$ defined in Section~\ref{sec:RS}, and the action $M_g$ can be applied to them. 
Then the action $M_g$ applied to $\d W$ yields (see Table~\ref{tab:renormalization})
\begin{equ}[eq:dW-remormalized]
M_g \d W = \d W - \frac{1}{2} C_0\, \<1d1d>
\end{equ}
and 
\begin{equ}[eq:dW-remormalized-2]
\CQ_{\leq 0} (M_g \d W)^2 =\CQ_{\leq 0} (\d W)^2.
\end{equ}
Furthermore, the functions inside the integrals on the right-hand side of \eqref{eq:fixed-point-proof} can be written up to order $0$ as
\begin{equs}
\CQ_{\leq 0} &\Bigl( (\d W)^{2} + 2 (\d W) (\d \Psi) + (\d \Psi)^{2} \Bigr) = \CQ_{\leq 0} \Bigl(\tilde w\, \one\, + \frac{1}{2}\, \<2d1d>\, + \frac{1}{4}\, \<2d2d1d>\,+ \Bigl(a_1(0) + \frac{\tilde w}{2}\Bigr)\, \<1d1d>\, \Bigr)^2 \\
&\qquad + 2 \CQ_{\leq 0} \Bigl(\tilde w\, \one\, + \frac{1}{2}\, \<2d1d>\, + \frac{1}{4}\, \<2d2d1d>\,+ \Bigl(a_1(0) + \frac{\tilde w}{2}\Bigr)\, \<1d1d>\, \Bigr)\,\<1d> + \<2d> \\
& = \tilde w^2 \, \one + \frac{1}{4}\, \<tree2>\, + \tilde w\, \<2d1d> + 2 \tilde w\, \<1d>\, + \<2d2d>\, + \frac{1}{2}\, \<tree1>\,+ (2 a_1(0) + \tilde w)\, \<1d2d> + \<2d>
\end{equs}
and 
\begin{equs}
\CQ_{\leq 0} \Bigl( a_1 \d W + a_1 \d \Psi \Bigr) = a_1 \Bigl(\tilde w\, \one\, + \frac{1}{2}\, \<2d1d>\, + \frac{1}{4}\, \<2d2d1d>\,+ \,\<1d>\, \Bigr).
\end{equs}
Applying the map $M_g$ to these functions (see Table~\ref{tab:renormalization}) and using \eqref{eq:dW-remormalized}-\eqref{eq:dW-remormalized-2}, we get
\begin{equs}
M_g &\CQ_{\leq 0} \Bigl( (\d W)^{2} + 2 (\d W) (\d \Psi) + (\d \Psi)^{2} \Bigr) = \CQ_{\leq 0} \Bigl( (\d W)^{2} + 2 (\d W) (\d \Psi) + (\d \Psi)^{2} \Bigr) \\
&\qquad - \Bigl(\frac{1}{4} C_2 \one + 2 C_0 \<1d> + C_0\, \<1d2d> + \frac{1}{2} C_0 \<2d1d> + \frac{1}{2} C_3 \one + (2 a_1(0) + \tilde w) C_0 \one + C_1 \one \Bigr) \\
&= \CQ_{\leq 0} \Bigl( (M_g \d W)^{2} + 2 (M_g \d W) (\d \Psi) + (\d \Psi)^{2} \Bigr) \\
&\qquad - C_0 \CQ_{\leq 0} (M_g \d W) - 2 C_0 \d \Psi - \Bigl(\frac{1}{4} C_2 + \frac{1}{2} C_3 + 2 a_1(0) C_0 + C_1 \Bigr) \one
\end{equs}
and 
\begin{equs}
M_g \CQ_{\leq 0} \Bigl( a_1 \d W + a_1 \d \Psi \Bigr) = \CQ_{\leq 0} \Bigl( a_1 M_g \d W + a_1 \d \Psi \Bigr)
\end{equs}
Applying the reconstruction operator $\CR_g$ to these functions and using \cite[Prop.~15.10]{FrizHairer}, we get 
\begin{equs}
&\CR_g M_g \CQ_{\leq 0} \Bigl( (\d W)^{2} + 2 (\d W) (\d \Psi) + (\d \Psi)^{2} \Bigr) \\
& = \Bigl( (\partial_x w)^2 + 2 \partial_x w \partial_x \psi + (\partial_x \psi)^2 \Bigr) - C_0 \partial_x w - 2 C_0 \partial_x \psi - \Bigl(\frac{1}{4} C_2 + \frac{1}{2} C_3 + 2 a_1(0) C_0 + C_1 \Bigr)
\end{equs}
and 
\begin{equs}
\CR_g M_g \CQ_{\leq 0} \Bigl( a_1 \d W + a_1 \d \Psi \Bigr) = \CQ_{\leq 0} \Bigl( a_1 \partial_x w + a_1 \partial_x \psi \Bigr),
\end{equs}
where $\partial_x w = \CR_g M_g \d W$. Hence, applying $\CR_g M_g$ to equation \eqref{eq:fixed-point-proof} we get exactly \eqref{eq:renormalized-equation} with the constants 
\begin{equ}
c_g^{(1)} = C_0, \qquad \qquad c_g^{(2)} = 2 C_0, \qquad\qquad c_g^{(3)} = \frac{1}{4} C_2 + \frac{1}{2} C_3 + 2 a_1(0) C_0 + C_1.
\end{equ}
Moreover, the function in the non-linearity in \eqref{eq:renormalized-equation} is locally Lipschitz continuous with respect to $\d w$ because it is a polynomial. This gives Assumption~8 in \cite{HairerMattingly}.

The map $F_i$ in \eqref{eq:fixed-pt-new} is twice differentiable as a map between the finite-dimensional vector spaces $\CQ_{< \gamma} \CU$ and $\CQ_{< \gamma_i} \CV$, because it is a polynomial. 
We can compute the first derivatives of these functions:
\begin{equs}
D F_0(W, \Psi) = \d W, \qquad D F_1(W, \Psi) &= \frac{1}{2} \d \Psi, \qquad D F_3(W, \Psi) = a_1, \\
D F_2(W, \Psi) = D F_4(W, \Psi) &= D F_5(W, \Psi) = 0.
\end{equs}
If $W$ and $J$ belong to $\SD_P^{\gamma, w}$ with values in $\CQ_{< \gamma} \CU$, then the map 
\begin{equ}
z \mapsto \bigl(DF_i(W(z), \Psi, z) \bigr) J(z)
\end{equ}
belongs to $\SD_P^{\gamma_i, w_i}$, and the map $(W, J) \mapsto DF_i(W, \Psi) J$ is strongly locally Lipschitz continuous, as follows from \cite[Prop.~6.12]{Regularity}. This yields Assumption~9 in \cite{HairerMattingly}.

The previous assumptions guarantee that the abstract problem \eqref{eq:fixed-pt-new} has a local (in time) solution. Moreover, we have differentiability of the solution map, which is required by the properties described in Section~\ref{sec:how-to-prove}. We have convergence of of the lifts of mollified driving noises and we understand the action of the renormalization group on the equation, which in turn yields convergence of the renormalized solutions. It is left to describe the shift introduced in Section~\ref{sec:how-to-prove}.

Since the definition of the model for the regularity structure does not depend on the boundary conditions, Assumptions~10 and 11 in \cite{HairerMattingly} are satisfied with the space $X_0 = \CC$ and the function $G = 1$, as shown in \cite[Thm.~5.3]{HairerMattingly} for a more general equation.

Let $h_t(x) = (\CR W + \Psi)(t,x)$ be the reconstructed solution of equation \eqref{eq:fixed-pt-new}, and let $\fP_t$ be its transition semigroup at time $t > 0$. The following result follows from the just stated properties of the involved objects and \cite[Thm.~4.8]{HairerMattingly}.

\begin{proposition}\label{prop:KPZ-Feller}
The semigroup generated by the open KPZ equation \eqref{eq:openKPZ} enjoys the strong Feller property.
\end{proposition}

\section{Proof of Theorem~\ref{thm:main}}
\label{sec:proof-main}

 The stationary measure $\mu$, described in Section~\ref{sec:measure}, lives on the space $\widetilde{\CC}^\alpha$ for any $\alpha \in (0,\frac{1}{2})$. In particular, this is a Borel measure on the Polish space.

\begin{proposition}\label{prop:support}
Assume that either $u + v = 0$ or $u + v >0$, $\min(u, v)>-1$. Then $\mu$ has full support on $\widetilde{\CC}^\alpha$ for any $\alpha \in (0, \frac{1}{2})$.
\end{proposition}

\begin{proof}
For $u + v >0$, $\min(u, v)>-1$, from the representation \eqref{eq:invariant-representation} of the invariant measure we see that the Radon-Nikodym derivative \eqref{eq:RN-deriv} is strictly
  positive. Since the law of the Brownian motion $W$ has a full support on $\CC^\alpha$, we conclude that the same if true for the stationary distribution of the open KPZ equation. 
  
  The case $u + v = 0$ is trivial, because then the stationary measure is a Brownian motion with a linear shift (see Theorem~\ref{thm:WASEP-measure}), which has a full support. 
\end{proof}

The solution of the open KPZ equation \eqref{eq:openKPZ} is a Markov process on the space $\CC^\alpha([0, 1])$ of $\alpha$-H\"{o}lder continuous function for any $\alpha \in (0,\frac{1}{2})$. Let $\fP_t$ be its transition semigroup, i.e., $\fP_t$ is a Markov operator defined as
\begin{equ}
(\fP_t \Psi)(f) = \E_{h_0 = f} [\Psi(h_t)],
\end{equ}
for any bounded measurable function $\Psi : \CC^\alpha([0, 1]) \to \R$.

Now, we can conclude uniqueness of the invariant measure. Proposition~\ref{prop:KPZ-Feller} implies that for every $t > 0$ the semigroup $\fP_t$ is strong Feller. Uniqueness of the stationary measure then follows readily from Proposition~\ref{prop:support} and Corollary~\ref{cor:uniqueness}, applied to the Markov operator $\fP_t$.

\bibliographystyle{Martin}
 \bibliography{refs}

\end{document}